\documentclass[11pt]{amsart}
\usepackage{amssymb,amsmath}
\def\Bbb{\mathbb}
\def\Cal{\mathcal}
\def\Dt{\partial_t}

\def\eb{\varepsilon}

\def\R {\mathbb{R}}
\def\N {\mathbb{N}}

\def\M {{\mathcal M}}

\def \pt {\partial_t}

\def\<{\left<}
\def\>{\right>}

\def\({\left(}
\def\){\right)}

\newtheorem{proposition}{Proposition}[section]
\newtheorem{theorem}[proposition]{Theorem}
\newtheorem{corollary}[proposition]{Corollary}
\newtheorem{lemma}[proposition]{Lemma}

\theoremstyle{definition}
\newtheorem{definition}[proposition]{Definition}

\newtheorem{remark}[proposition]{Remark}

\numberwithin{equation}{section}

\def \au {\rm}
\def \ti {\it}
\def \jou {\rm}

\def \no#1#2#3 {{\bf #1} (#3), #2.}
\def \eds#1#2#3 {#1, #2, #3.}

\title[Inertial manifold for Cahn-Hilliard equation] {Inertial manifolds for the 3D Cahn-Hilliard equations with periodic boundary conditions }
\author[A.Kostianko and S.Zelik]{ Anna Kostianko${}^1$ and Sergey Zelik${}^1$}
\address{${}^1$
University of Surrey, Department of Mathematics,
Guildford, GU2 7XH, United Kingdom.}
\email{a.kostianko@surrey.ac.uk; s.zelik@surrey.ac.uk}
\begin{document}
\begin{abstract} The existence of an inertial manifold for the 3D Cahn-Hilliard equation with periodic boundary conditions is verified using the proper extension of the so-called spatial averaging principle introduced by G. Sell and J. Mallet-Paret. Moreover, the extra regularity of this manifold is also obtained.
\end{abstract}
\thanks{This work is partially supported by  the grant RSF 14-41-00044 of RSF and the grant 14-01-00346 of RFBR}   
\maketitle
\tableofcontents
\section{Introduction}\label{s0}

It is believed that  the dynamics generated by dissipative PDEs in bounded domains is typically finite-dimensional. The latter means that despite the infinite-dimensionality of the initial phase space, the limit dynamics, say, on the so-called global attractor can be effectively described by finitely many parameters which satisfy a system of ODEs -- the so-called inertial form of the dissipative PDEs  considered, see \cite{BV,CV,MirZe,28,tem} and references therein. This reduction clearly works when the underlying PDE possesses an {\it inertial} manifold (IM) that is a finite-dimensional invariant $C^1$-smooth manifold with exponential tracking property. Then the desired inertial form can be constructed just by restricting the considered PDE to the invariant manifold, see \cite{FST,mik,rom-man}. However, the existence of an inertial manifold requires rather strong {\it spectral gap} assumptions which are usually satisfied only for the parabolic equations in the space dimension one and, despite a big permanent interest and many results obtained in this direction, the  finite-dimensional reduction for the case where the IM does not exist remains unclear and there are even some evidence that the dissipative dynamics may be {\it infinite-dimensional} in this case, see \cite{EKZ,sell-counter,Zel} and reference therein.
\par
It is also known that the above mentioned spectral gap assumptions are sharp and cannot be relaxed/removed at least on the level of the abstract functional models associated with the considered PDE, see \cite{EKZ,mik,rom-man}. However, an IM may exist for some concrete classes of PDEs even when the spectral gap condition is violated. The most famous example is a scalar reaction-diffusion equation
\begin{equation}\label{0.RDE}
\Dt u=\Delta_x u-f(u)
\end{equation}
on a 3D torus $x\in[-\pi,\pi]^3$. Here the spectral gap condition reads
\begin{equation}\label{0.gapRDE}
\lambda_{N+1}-\lambda_N>2L,
\end{equation}
where $\lambda_1\le\lambda_2\le\cdots$ are the eigenvalues of the minus Laplacian on a torus enumerated in the non-decreasing order and $L$ is a Lipschitz constant of the nonlinearity $f$. The eigenvalues of the Laplacian  are all natural numbers which can be presented as sums of 3 squares and by the Gauss theorem, there are no gaps of length more than 3 in the spectrum, so the spectral gap condition  clearly fails if the Lipschitz constant $L$ is large enough. Nevertheless, the corresponding IM can be constructed (for all values of the Lipschitz constant $L$) using the so-called {\it spatial averaging} principle introduced in \cite{mal-par}.
\par
One more example is the 1D reaction-diffusion-advection problem
\begin{equation}\label{0.RDEA}
\Dt u=\partial^2_x u+\partial_x f(u)-g(u),\ \ x\in[0,\pi],\ u\big|_{x=0}=u\big|_{u=\pi}=0,
\end{equation}
where the spectral gap condition is also not satisfied initially, but is satisfied after the proper change of the dependent variable $u$, see \cite{Zel1} and also \cite{rom-th,rom-th1} where the Lipschitz continuous inertial form is constructed.
\par
Although the spatial averaging principle has been used to get the IM for reaction-diffusion equations in some non-toroidal domains, see \cite{kwean}, to the best of our knowledge, it has been never applied before to the equations different from the scalar reaction-diffusion ones. The aim of the paper is to cover this gap by extending the method to the so-called Cahh-Hilliard equation on a 3D torus. To be more precise, we consider the following 4th order parabolic problem:
\begin{equation}\label{0.CH}
\Dt u+\Delta_x(\Delta_x u-f(u))=0,\ \ \partial_n u\big|_{\partial\Omega}=\partial_n\Delta_xu\big|_{\partial\Omega}=0,\ \ u\big|_{t=0}=u_0,
\end{equation}
where $\Omega$ is a bounded 3D domain and $f(u)$ is a given non-linear interaction function, see \cite{CH,Ell,No1} and the references therein concerning the physical background of this equation. We also assume that this function satisfies some standard dissipativity assumptions, so the associated semigroup possesses a global attractor $\mathcal A$ which is bounded in $H^2(\Omega)\subset C(\Omega)$, see e.g., \cite{CMZ,MirZe,tem} for more details. By this reason, without loss of generality, we may assume from the very beginning that the function $f:\R\to\R$ is {\it globally} bounded and is {\it globally} Lipschitz continuous with the Lipschitz constant $L$. It worth mentioning that this equation possesses a mass conservation law
\begin{equation}\label{0.int}
\frac d{dt}\<u(t)\>=0,\ \ \<u\>=\frac1{|\Omega|}\int_\Omega u(x)\,dx,
\end{equation}
so we assume from now on that $\<u(t)\>=\<u(0)\>=0$. Thus, the natural phase space of the problem is
\begin{equation}\label{0.phase}
 H^{-1}:=H^{-1}(\Omega)\cap\{\<u_0\>=0\}.
\end{equation}
Note that the spectral gap condition for the IM existence for equation \eqref{0.CH} reads
\begin{equation}\label{0.CHgap}
\frac{\lambda_{N+1}^2-\lambda_N^2}{(\lambda_{N}^2)^{1/2}+(\lambda_{N+1}^2)^{1/2}}=\lambda_{N+1}-\lambda_N>L.
\end{equation}
This condition is clearly satisfied for 1D domains only, in the 2D case it is still an open problem whether or not the spectral gaps of arbitrary width exist for any/generic domains $\Omega$ although it will be so for some special domains like 2D sphere or 2D torus. In these cases the construction of the IM is straightforward, see \cite{BM,tem} for more details. However, it is extremely unlikely that the spectral gap condition is satisfied for more or less general 3D domains (in a fact, we know the only example of a 3D sphere
 where it is true). In particular, it obviously fails for the case of a 3D torus $\Omega=\mathbb T^3=[-\pi,\pi]^3$ (endowed by periodic boundary conditions), therefore, the problem of finding the IM for the 3D Cahn-Hilliard equation with periodic boundary conditions becomes non-trivial and to the best of our knowledge, has been not considered before.

\par
The next theorem gives the main result of the paper.
\begin{theorem}\label{Th0.main} For infinitely many values of  $N\in\mathbb N$ there exists an $N$-dimensional IM $\mathcal M_N$ for the Cahn-Hilliard problem \eqref{0.CH} with periodic boundary conditions which is a graph of a Lipschitz continuous function over the $N$-dimensional space spanned by the first $N$ eigenvectors of the Laplacian. Moreover, this function is $C^{1+\eb}$-smooth for some small  $\eb=\eb(N)>0$ and the manifold
possesses the so-called exponential tracking property, see Section \ref{s2} for the details.
\end{theorem}

The paper is organized as follows.
\par
In Section \ref{s1}, we consider the functional model related with the problem considered and prepare some technical tools which will be used later.
\par
In Section \ref{s2}, for the reader convenience, we remind the invariant cone and squeezing property as well as give the proof of the IM existence theorem for our class of equations under the assumption that the cone and squeezing property are satisfied (following mainly  \cite{Zel}).
\par
In Section \ref{s3}, we reformulate the cone and squeezing property in a more convenient form of a single differential inequality and derive some kind of normal hyperbolicity (dominated splitting) estimates which are necessary to verify the smoothness of an IM.
\par
In Section \ref{s3.5}, we verify that the constructed manifold is $C^{1+\eb}$-smooth if the nonlinearity is smooth enough. This improves the result of \cite{mal-par} even on the level of reaction-diffusion equations where only $C^1$-smoothness has been verified.
\par
The abstract form of spatial averaging principle has been stated in Section \ref{s4} and the existence of the IM is verified under the assumption that this principle holds.
\par
Finally, in Section \ref{s5}, we verify this principle for the case of the Cahn-Hilliard equation on a 3D torus and finish the proof of the main Theorem \ref{Th0.main}.

\section{Preliminaries}\label{s1}

We consider the following equation:
\begin{equation}\label{main_eq}
\pt u + A^2 u + A F(u) = 0,\ \ u|_{t=0}=u_0,
\end{equation}
where $A: D(A) \to H$ is a linear self-adjoint positive operator with compact inverse, $D(A)$ is the domain of the operator $A$, and non-linearity $F: H \to H$ is a globally Lipschitz with Lipschitz constant $L$ and globally bounded, i.e.,
\begin{equation}\label{1.lip}
1.\ \ \|F(u)\|_H \le C, \ \ u\in H,  \quad \quad 2.\ \  \|F(u_1) - F(u_2)\| \le L \|u_1 - u_2\|, \ \ u_1, u_2 \in H.
\end{equation}
It is well known that problem \eqref{main_eq} is globally well-posed and generates a non-linear semigroup $S(t)$ in $H$.

From Hilbert-Schmidt theorem we conclude that the operator $A$ possesses the complete orthonormal system of eigenvectors $\{e_n\}_{n=1}^{\infty}$ in $H$ which corresponds to eigenvalues $\lambda_n $ numerated in the non-decreasing way:
\begin{equation}
A e_n = \lambda_n e_n,  \ \ \ 0<\lambda_1 \le \lambda_2 \le \lambda_3\le ...
\end{equation}
and due to the compactness of $A^{-1}$, we have $\lambda_n \to \infty$ as $n \to \infty$.

Thus, we may represent $u$ in the form:
\begin{equation}
u = \sum_{n=1}^{\infty}u_n e_n,  \ \ u_n = (u,e_n).
\end{equation}
Then, as usual, the normed spaces $H^s:= D(A^{s/2})$, $s\in \R_+$, is defined as follows
\begin{equation}
H^s=\bigg\{u\in H:\ \|u \|_{H^s}^2 = \sum_{n=1}^{\infty}\lambda_n^{s}u_n^2<\infty\bigg\}.
\end{equation}
For $s<0$ such defined space $H^s$ is not complete. Thus for negative $s$ we define $H^s$ as completion of $H$ with respect to corresponding norm $\|\cdot\|_{H^s}$.
Let us introduce the orthoprojector to the first $N$ Fourier modes:
\begin{equation}
P_N u:= \sum_{n=1}^N u_n e_n
\end{equation}
and denote by $Q_N = Id - P_N$, $H_+:= P_N H$ and $H_-:= Q_N H$. Obviously, the following estimates are valid:
\begin{equation}\label{est_for_H_+/-}
\begin{cases}
(Au,u) \le \lambda_N \|u\|^2_H,\  u \in H_+; \\
(Au,u) \ge \lambda_{N+1} \| u\|^2_H,\  u \in D(A^{1/2})\cap H_-.
\end{cases}
\end{equation}
Throughout the work we will use notations  $u_+ := P_N u$ and $u_- := Q_N u$ for given element $u \in H$.
\par
The next proposition collects the standard dissipativity and smoothing properties of the solution semigroup associated with equation \eqref{main_eq}, see \cite{hen,Zel,tem} for more details.
\begin{proposition}\label{Prop1.trivial}
Let the non-linearity $F$ and operator $A$ satisfy the above assumptions. Then, problem \eqref{main_eq} is uniquelly solvable for any $u_0\in H^{-1}$ and, therefore, the solution semigroup $S(t):H^{-1}\to H^{-1}$ is well-defined. Moreover,
the following properties hold for any solution $u(t)$ of problem \eqref{main_eq}:
\par
1. Dissipativity in $H^s$ for $s \in [-1,2]$:
\begin{equation}\label{diss}
\|u(t)\|_{H^s} \le C e^{-\gamma t}\|u(0)\|_{H^s} + R_*,
\end{equation}
where $C$, $\gamma$ and $R_*$ are some positive constants which are independent of the solution $u$ and $t$;

2. Smoothing property:
\begin{equation}\label{smooth}
\|u(t)\|_{H^2} \le C t^{-1} \|u(0)\|_H + R_0, \ \ t>0,
\end{equation}
where $C$ and $R_0$ are independent of $u$ and $t$
\par
3. Dissipativity of the $Q_N$ component:
\begin{equation}\label{diss_Q_N}
\|Q_N u(t)\|_{H^{2-\kappa}} \le C e^{-\gamma t}\|Q_N u(0)\|_{H^{2-\kappa}} + R_\kappa,
\end{equation}
for all $N \in \N$ and $\kappa \in (0, 3]$. Here $C$, $\gamma$ and $C_\kappa$ are independent of $N$, $u$ and $t$.
\end{proposition}
We are now ready to give the key definition of the paper, namely, to define the inertial manifold (IM) associated with the Cahn-Hilliard equation.

\begin{definition}\label{Def1.IM}
The set $\M\subset H$ to be called an inertial manifold for problem \eqref{main_eq} if the following conditions are satisfied:
\par
1. The set $\M$ is invariant with respect to the solution semigroup $S(t)$, i.e. $S(t) \M = \M$;
\par
2. It can be presented as a graph of a Lipschitz continuous function $\Phi: H_{+} \to H_{-} $:
\begin{equation}
\M:= \{u_+ + \Phi(u_+), u_+ \in H_+\};
\end{equation}
\par
3. The exponential tracking property holds, i.e., there exist positive constants $C$ and $\alpha$ such that for every $u_0 \in H^{-1}$ there is $v_0 \in \mathcal M$ such that
\begin{equation}\label{1.llip}
\| S(t)u_0 - S(t)v_0\|_{H^{-1}} \le C e^{-\alpha t}\| u_0 - v_0\|_{H^{-1}},\quad t\geq0.
\end{equation}
\end{definition}
As usual, to verify the existence of the IM, we will use invariant cones method. Namely,
introduce the following quadratic form in $H^{-1}$:
\begin{equation}\label{quadratic}
V( \xi)=\|Q_N \xi \|^2_{H^{-1}}- \|P_N \xi \|^2_{H^{-1}},\ \xi \in H^{-1}
\end{equation}
and set $K^+=K^+(N):=\{\xi\in H^{-1},\ \ V(\xi)\le0\}$ to be the associated cone.

\begin{definition}\label{Def1.con-squeez} Let the above assumptions hold.
We say that equation \eqref{main_eq} possesses the cone  property (invariance of the cone $K^+$) if
\begin{equation}\label{cone_con}
\xi_1 - \xi_2 \in K^+ \Rightarrow S(t)\xi_1 - S(t)\xi_2 \in K^+, \text{ for all } t\ge 0,
\end{equation}
where $\xi_1, \xi_2 \in H^{-1}$ and $S(t)$ is a solution semigroup associated with \eqref{main_eq}.
\par
Analogously, we say that \eqref{main_eq} possesses the squeezing property if there exists positive $\gamma$ and $C$ such that
\begin{equation}\label{squeez_pr}
S(T)\xi_1- S(T)\xi_2 \not \in K^+ \Rightarrow \|S(t)\xi_1 - S(t)\xi_2\|_{H^{-1}} \le C e^{-\gamma t} \|\xi_1 - \xi_2\|_{H^{-1}},\  t\in [0, T].
\end{equation}
\end{definition}

\section{Invariant cones, squeezing property and inertial manifolds}\label{s2}

The aim of this section is to remind the reader how to construct an IM based on the cone and squeezing property, see \cite{fen,FST,mal-par,rom-man} for more details. So, the main result of the section is the following theorem.
\begin{theorem}\label{th_in_man}
Let the non-linearity $F$ be globally Lipschitz and globally bounded, see \eqref{1.lip} and let, in addition, the solution semigroup $S(t)$ associated with equation \eqref{main_eq} satisfies the cone and squeezing properties \eqref{cone_con} and \eqref{squeez_pr} for some $N\in\mathbb N$, see Definition \ref{Def1.con-squeez}. Then equation \eqref{main_eq} possesses an N-dimensional inertial manifold in the sense of Definition \ref{Def1.IM}.
\end{theorem}
\begin{proof}
{\it Step 1.} Let us consider the following boundary value problem:
\begin{equation} \label{aux_eq}
\pt u + A^2 u + A F(u) = 0, \ \ P_N u|_{t=0}=u_0^+, \ \ Q_N u |_{t= - T} = 0.
\end{equation}
We claim that it has a unique solution for any $T>0$ and any $u_0^+\in H^+$. Indeed, introduce the map $G_T: H^+ \to H^+$ by the following rule:
\begin{equation}
G_T(w) = P_N S(T)w, \ \ w\in H^+,
\end{equation}
where $S(t)$ is a solution operator of problem \eqref{main_eq}. Obviously this map is continuous. We want to prove that this map is invertible. Indeed, let
$u_1(t)$, $u_2(t)$ be two solutions of the problem \eqref{aux_eq} (with different initial data $u_1(-T)$ and $u_2(-T)$ belonging to $H_+$). Then, their difference $v(t)=u_1(t)-u_2(t)$
 lies at the cone $K^+$ at the moment $t = -T$. Thus, from the cone property we conclude that
\begin{equation}
u_1(t) - u_2(t) \in K^+, \ \ t\in [-T, 0].
\end{equation}
The next lemma is the main technical tool for verifying the one-to-one property.
\begin{lemma}\label{Lem2.new} Let the above assumptions hold. Then, the following estimate hold for the solutions $u_1(t)$ and $u_2(t)$:
\begin{equation}\label{2.con-lip}
\|u_1(-T)- u_2(-T)\|^2_{H^{-1}} \le C e^{\alpha T}\|P_N u_1(0)- P_N u_2(0)\|^2_{H^{-1}},
\end{equation}
for some constants $C$ and $\alpha$ which are independent of $u_i$.
\end{lemma}
\begin{proof}[Proof of the lemma] Since $v(t)\in K^+$, we have the estimate
\begin{equation}\label{2.con-est}
\|v_-(t)\|^2_{H^{-1}}\le\|v_+(t)\|^2_{H^{-1}}.
\end{equation}
Multiplying now the equation for the difference $v$ by $A^{-1}v_+$ and using that the nonlinearity is globally Lipschitz, we have
$$
\frac d{dt}\|v_+(t)\|^2_{H^{-1}}+\|v_+(t)\|^2_{H^1}=(F(u_1)-F(u_2),v_+)\ge-L\|v_+\|^2_H-L\|v_+\|_H\|v_-\|_H.
$$
Integrating this inequality over $s\in(t,0)$ and using the interpolation between $H^1$ and $H^{-1}$ together with estimate \eqref{2.con-est}, we end up with
\begin{equation}\label{2.almost}
\|v_+(t)\|^2_{H^{-1}}\le\|v_+(0)\|^2_{H^{-1}}+ C\int_t^0\|v_+(s)\|^2_{H^{-1}}\,ds+\frac12\int_t^0\|v_-(s)\|^2_{H^1}\,ds.
\end{equation}
To estimate the last term in the right-hand side, we multiply the equation for $v$ by $A^{-1}v_-$ and integrate over $s\in(t,0)$. With the help of \eqref{2.con-est} again, this gives
\begin{multline}
\int_t^0\|v_-(s)\|^2_{H^1}\,ds=\frac12\(\|v_-(0)\|^2_{H^{-1}}-\|v_-(t)\|^2_{H^{-1}}\)+\int_t^0(F(u_1)-F(u_2),v_-)\,ds
\le\\\le\frac12\|v_+(0)\|^2_{H^{-1}}
+L\int_t^0\|v\|^2_H\le\frac12\|v_+(0)\|^2_{H^{-1}}+\frac12\int_t^0\|v_-(s)\|^2_{H^1}\,ds+C\int_t^0\|v_+(s)\|^2_{H^{-1}}\,ds.
\end{multline}
Inserting the last estimate into the RHS of \eqref{2.almost}, we finally arrive at
\begin{equation}\label{2.almost1}
\|v_+(t)\|^2_{H^{-1}}\le 2\|v_+(0)\|^2_{H^{-1}}+ C\int_t^0\|v_+(s)\|^2_{H^{-1}}\,ds
\end{equation}
and the Gronwall inequality finishes the proof of the lemma.
\end{proof}

We are now ready to finish the first step of the proof of the theorem. Indeed,
since $u_1(t)$ and $u_2(t)$ were chosen arbitrary, then we conclude that the map $G_T: H_+ \to H_+$ is injective. Consequently, by the Brouwer invariance of domain theorem, $G_T(H_+)$ is open. Moreover, estimate \eqref{2.con-lip} guarantees that the sequence $w_n\in H_+$ is bounded if $G_T(w_n)$ is bounded. Then, since $H_+$ is finite-dimensional, $G_T(H_+)$ is also closed by compactness arguments. Thus, $G_T(H_+)=H_+$ and $G_T$ is a (bi-Lipschitz) homeomorphism on $H_+$. Therefore,  $G_T(v) = u_0^+$ has a unique solution for all $u_0^+ \in H_+$. Then, obviously $u(t) = S(t+T)G_T^{-1}(u_0^+)$ solves \eqref{aux_eq} and the first step is completed.

{\it Step 2.} Let $u_{T, u_0^+}$ be the solution of the boundary value problem \eqref{aux_eq}. We claim that for all $t\le0$, there exists a limit
\begin{equation} \label{lim}
u_{u_0^+}(t) = \lim_{T \to \infty} u_{T,u_0^+}(t)
\end{equation}
 which solves problem \eqref{aux_eq} with $T =\infty$.

Indeed,
 since solution of the problem \eqref{aux_eq} starts from $u_-(-T) = 0$, according to Proposition \ref{Prop1.trivial}, we have:
\begin{equation}\label{Q_est}
\|Q_N u_{T, u_0^+}(t)\|^2_{H^{2-\kappa}} \le \tilde{C_\kappa}
\end{equation}
for all $T \ge 0$ and $u_0^+ \in H$ and $\kappa\in(0,3]$. In particular, the choice $\kappa=3$ gives the control of the $H^{-1}$-norm.
 \par

 Let us introduce the following notations $u_i(t) := u_{T_i, u_0^+}(t)$ and $v(t) = u_1(t) - u_2(t)$. Then we know that at the moment $t=0$ we have $v_+(0) = 0$ and consequently $v(0) \not \in K^+$. By the cone property \eqref{cone_con}
\begin{equation}\label{not_in_K}
v(t) \not \in K^+ \text{ for all } t \in [-T, 0] \text{, where }T = \min \{T_1, T_2\}.
\end{equation}
Thus, using squeezing property \eqref{squeez_pr} we get:
\begin{equation} \label{sq_pr_act}
\|v(t)\|_{H^{-1}} \le C e^{- \gamma(t+T)}\|v(-T)\|_{H^{-1}}.
\end{equation}
Due to \eqref{not_in_K}, we have $\|v_+(-T)\|_{H^{-1}}\le \|v_-(-T)\|_{H^{-1}}$ and consequently thanks to \eqref{Q_est},
\begin{equation} \label{2.estt}
\|v(t)\|_{H^{-1}} \le C_1 e^{- \gamma(t+T)}.
\end{equation}
 Thus, $u_{T_i, u_0^+}$ is a Cauchy sequence in $C_{loc}((-\infty,0), H^{-1})$. Consequently, there exists limit \eqref{lim} and $u_{u_0^+}$ is a backward solution of the problem \eqref{main_eq}.

{\it Step 3.} Let us define a set $ \mathcal{N} \subset C _{loc}(\R, H)$ as the set of all solutions of the problem \eqref{main_eq} obtained as a limit \eqref{lim}. Then, by the construction, $\mathcal{N}$ is invariant with respect to the solution semigroup $S(t)$, i.e.
\begin{equation}
S(h)\mathcal{N} = \mathcal{N}, \ \ (S(h)u)(t) = u(t+h), \ \ h \in \R.
\end{equation}
Consider function $\Phi: H_+ \to H_-$ acting by the rule:
\begin{equation}
\Phi(u_0)= Q_N u(0), \ \ u(t) \in \mathcal{N}, \ \ P_N u(0) = u_0.
\end{equation}
Indeed, according to Steps 1 and 2, the trajectory $u\in\mathcal N$ exists for any $u_0\in H_+$. Moreover, as not difficult to see by approximating the solutions $u_1,u_2\in\mathcal N$ by the solutions of the boundary value problem \eqref{aux_eq}, that
\begin{equation} \label{in_K}
u_1(t) - u_2(t) \in K^+, \ \ t\in \R
\end{equation}
for any $u_1,u_2\in\mathcal N$.
Therefore, $\Phi(u_0)$ is well defined and Lipschitz continuous.
Thus, it remains to note that manifold
\begin{equation}
\mathcal{M} := \big\{u_0 + \Phi(u_0): u_0 \in H^+\big\}
\end{equation}
is invariant with respect to $S(t)$ which is a direct consequence of the invariance of $\mathcal{N}$.

{\it Step 4.} In order to prove that $\mathcal{M}$ is a desired inertial manifold, it remains to show that exponential tracking property holds. Let $u(t)$ , $t \ge 0$, be a forward solution of the problem \eqref{main_eq} and $u_T(t) \in \mathcal{N}$, $T >0$, be the solution of \eqref{main_eq} which belongs to the manifold $\mathcal{M}$ such that
\begin{equation}
P_N u (T) = P_N u_T (T).
\end{equation}
Then, due to the cone property \eqref{cone_con} $u(t) - u_T(t) \not \in K^+$ as $t \in [0,T]$ and consequently
\begin{equation}
\|P_N(u(t) - u_T(t))\|^2_{H^{-1}} \le \|Q_N (u(t) - u_T(t))\|^2_{H^{-1}}, \ \ t \in [0,T].
\end{equation}
From \eqref{Q_est} we know that $\|Q_N u_T(t)\|_{H^{-1}}$ is uniformly bounded with respect to $T$. Since $u(t)$ is also bounded in $H^{-1}$, see Proposition \ref{Prop1.trivial}, we conclude that the functions
$u_T(t)\in\mathcal M$ are uniformly bounded as $T \to \infty$ in the norm $H^{-1}$  for all fixed $t \in [0,T]$. On the other hand, using squeezing property \eqref{squeez_pr} we obtain:
\begin{equation} \label{exp_tr}
\|u(t) - u_T(t)\|_{H^{-1}} \le C e^{-\gamma t} \|u(0) - u_T(0)\|_{H^{-1}}, \ \ t\in [0,T].
\end{equation}
Since the manifold $\mathcal M$ is finite-dimensional and $u_{T}(0)$ is uniformly bounded, we may assume
 without loss of generality that $u_{T_n}(0)\to\tilde u(0)$ for some $T_n\to\infty$ and $\tilde u(0)\in\mathcal M$. Then, according to \eqref{1.llip} and Lemma \ref{Lem2.new}, $u_{T_n}\to\tilde u\in\mathcal N$ in $C_{loc}(\R,H^{-1})$. Passing to the limit $n\to\infty$ at  \eqref{exp_tr} we see that
\begin{equation} \label{2.exptr}
\|u(t) - \tilde u(t)\|_{H^{-1}} \le C e^{-\gamma t} \|u(0) - \tilde u(0)\|_{H^{-1}}, \ \ t\ge0.
\end{equation}
Thus, the exponential tracking is verified and the theorem is proved.
\end{proof}

\begin{remark}\label{Rem2.smooth} Note that, according to Lemma \ref{Lem2.new} and the property \eqref{in_K}, any two solutions $u_1,u_2\in\mathcal N$ satisfy the backward Lipschitz continuity
\begin{equation}\label{est_tech}
\|u_1(-T)-u_2(-T)\|_{H^{-1}}\le Ce^{\alpha T}\|u_1(0)-u_2(0)\|_{H^{-1}},\ \ T\ge0,
\end{equation}
where the positive constants $C$ and $\alpha$ are independent of $u_1,u_2\in\mathcal N$ and $T\ge0$.
\par
One more simple but important observation is that the above given proof uses the cone and squeezing property not for all $u_1,u_2\in H^{-1}$ but only for those which satisfy estimate \eqref{Q_est}. By this reason, we actually need to verify the cone and squeezing property only for $u_1,u_2\in H^{2-\kappa}$ for some $\kappa\in(0,3]$ such that
\begin{equation}\label{2.important}
\|Q_N u_i\|_{H^{2-\kappa}}\le C_\kappa,\ \ i=1,2,
\end{equation}
where $C_\kappa$ is some constant depending only on $\kappa$. Indeed, this inequality is automatically satisfied for all trajectories involving into the construction of the set $\mathcal N$ and the associated inertial manifold $\mathcal M$ (due to estimate \eqref{diss_Q_N} and the fact that $Q_Nu(-T)=0$ in the boundary value problem \eqref{aux_eq}). The exponential tracking a priori holds for all trajectories starting from $u(0)\in H^{-1}$, however, due to the smoothing property (see Proposition \ref{Prop1.trivial}), it is sufficient to verify it only for $u(0)$ satisfying \eqref{2.important}. This observation plays a crucial role in the construction of a special cut-off for the spatial averaging method, see below.
\end{remark}


\section{Invariant cones and normal hyperbolicity}\label{s3}

In this section, we reformulate the cone and squeezing property in a more convenient (at least for our purposes) form of a single differential inequality on the trajectories of \eqref{main_eq} and its equation of variations and state a number of technical results related with the invariant cones and exponential dichotomy/normal hyperbolicity for the trajectories of the corresponding equation of variations. These results will be used in the next section for establishing the smoothness of the IM. We further assume that the nonlinearity $F(u)$ is at least Gateaux differentiable and there exists a derivative $F'(u)\in\mathcal L(H,H)$ for any $u\in H$. Then, obviously
\begin{equation}\label{3.1}
\|F'(u)\|_{\mathcal L(H,H)}\le L
\end{equation}
and we also assume that the following integral version of the mean value theorem holds:
\begin{equation}\label{3.2}
F(u_1)-F(u_2)=l_{u_1,u_2}(u_1-u_2),\ \ l_{u_1,u_2}:=\int_0^1F'(su_1+(1-s)u_2)\,ds.
\end{equation}
Then, the difference $v(t)=u_1(t)-u_2(t)$ of any two solutions $u_1(t)$ and $u_2(t)$ of the Cahn-Hilliard problem \eqref{main_eq} solves the following linear equation:
\begin{equation}\label{3.dif}
\Dt v+A(Av+l_{u_1(t),u_2(t)}v)=0,
\end{equation}
however, it will be more convenient for us to study more general linear equations
\begin{equation}\label{3.gen}
\Dt v+A(Av+l(t)v)=0,
\end{equation}
where $l\in L^\infty(\R,\mathcal L(H,H))$ satisfies $\|l(t)\|_{\mathcal L(H,H)}\le L$ for all $t\in\R$.
\begin{definition}\label{Def3.diffcon} We say that equation \eqref{3.gen} satisfies the strong cone condition (in a differential
 form) if there exist a positive number $\mu$ and a function $\alpha:\R\to\R$ such that
 \begin{equation}\label{3.squeez}
 0<\alpha_-\le \alpha(t)\le \alpha_+
 \end{equation}
 and for any solution $v(t)$, $t\in[S,T]$, $S<T$, of problem \eqref{3.gen}, the following inequality holds:
 \begin{equation}\label{3.dcone}
 \frac d{dt}V(v(t))+\alpha(t)V(v(t))\le-\mu\|v(t)\|^2_H
 \end{equation}
 for all $t\in[S,T]$. Here and below $V$ is the quadratic form defined by \eqref{quadratic}.
\end{definition}
As easy to see from inequality \eqref{3.dcone}, the above assumptions guarantee that the cone $K^+$ is invariant with respect to the evolution generated by equation \eqref{3.gen}. Moreover, as will be shown below, the squeezing property is also incorporated in our version of the strong cone condition (due to the strict positivity of the exponent $\alpha(t)$), but we first need to remind some elementary properties of the introduced strong cone condition. We start with reformulating it in a pointwise form applicable to the non-homogeneous form of equation \eqref{3.gen}.
\begin{lemma}\label{Lem3.point} The strong cone condition in the differential form is equivalent to the following condition:
\begin{equation}\label{8}
-2 (l(t)w,w_--w_+)-2(A w,w_--w_+)+\alpha(t)(\|w_-\|^2_{H^{-1}}-\|w_+\|^2_{H^{-1}})\le-\mu\|w\|^2_H
\end{equation}
for all $t\in\R$ and $w\in H^1$.
\end{lemma}
\begin{proof} Indeed, let condition \eqref{8} be satisfied. Then, multiplying equation \eqref{3.gen} by the expression $A^{-1}(v_-(t)-v_+(t))$ and using equation \eqref{3.gen}, we have
\begin{equation}\label{3.pro}
\frac d{dt}V(v(t))=2(\Dt v(t),v_-(t)-v_+(t))=-2(Av(t),v_-(t)-v_+(t))-2(l(t)v(t),v_-(t)-v_+(t)).
\end{equation}
Estimating the right-hand side of this inequality by \eqref{8} with $w=v(t)$, we end up with the desired cone inequality \eqref{3.dcone}. Vise versa, let the cone condition \eqref{3.dcone} hold and let $w\in H^1$ and $t_0\in\R$ be arbitrary. Let us  consider the solution $v(t)$ of equation \eqref{3.gen} satisfying $v(t_0)=w$. Using then the cone condition \eqref{3.dcone} with $t=t_0$ and formula \eqref{3.pro} for the derivative of the quadratic form $V$, we end up with the desired inequality \eqref{8} with $t=t_0$. Thus, the lemma is proved.
\end{proof}
\begin{corollary}\label{Cor3.non} Assume that equation \eqref{3.gen} possesses the strong cone property in the sense of Definition \ref{Def3.diffcon}. Then, for any solution $v(t)$ of the non-homogeneous equation
\begin{equation}\label{9}
\Dt v+A^2 v + A\big( l(t)v\big)+ A h(t) = 0
\end{equation}
with $h\in L^\infty(\R,H^{-1})$ the following analogue of \eqref{3.dcone} holds:
\begin{equation}
\label{10}
\frac d{dt} V(v(t))+\alpha(t) V(v(t))\le -\mu\|v(t)\|^2_H - 2(h(t),v_-(t)-v_+(t)).
\end{equation}
\end{corollary}
Indeed, analogously to \eqref{3.pro}, but using the non-homogeneous equation \eqref{9}, we have
\begin{equation}\label{3.pro1}
\frac d{dt}V(v(t))=-2(Av(t),v_-(t)-v_+(t))-2(l(t)v(t),v_-(t)-v_+(t))-2(h(t),v_-(t)-v_+(t))
\end{equation}
and estimating the terms in the right-hand side of this inequality with the help of \eqref{8} with $w=v(t)$, we end up with the desired inequality \eqref{10}.
\par
At the next step, we show that the strong cone condition is robust with respect to perturbations of the cone and this will give us the main technical tool for proving the smoothness of the IM. Namely, for any $\eb\in\R$, we define
\begin{equation}\label{veb}
V_\eb(\xi):=\eb\|\xi\|^2_{H^{-1}}+V(\xi)=(1+\eb)\|\xi_-\|_{H^{-1}}^2-(1-\eb)\|\xi_+\|^2_{H^{-1}}=(1+\eb)V(\xi)+2\eb\|\xi_+\|^2_{H^{-1}}.
\end{equation}
\begin{lemma}\label{Lem3.robust} Let equation \eqref{3.gen} satisfy the strong cone condition in the differential form. Then, there exists $\eb_0>0$ such that, for any $0<\eb<\eb_0$, the following inequalities
\begin{equation}\label{3.difeq}
\frac d{dt} V_\eb(v(t))+(\alpha(t)+\frac12\lambda_1\mu)V_\eb(v(t))\le 0,\ \ \frac d{dt} V_{-\eb}(v(t))+(\alpha(t)-\frac12\lambda_1\mu)V_{-\eb}(v(t))\le 0
\end{equation}
hold for any solution $v(t)$ of equation \eqref{3.gen}.
\end{lemma}
\begin{proof} Let us check the first inequality of \eqref{3.difeq}. Multiplying equation \eqref{3.gen} by $2\eb A^{-1}v$ and using the Lipschitz continuity, we have
\begin{equation}
\eb\frac d{dt}\|v(t)\|^2_{H^{-1}}\le 2\eb L\|v(t)\|^2_{H}.
\end{equation}
Taking a sum of this inequality with  \eqref{3.dcone}, using \eqref{3.squeez} and fixing $\eb>0$ to be so small that
$$
(2 L+\alpha_+\lambda_1^{-1})\eb\le\frac14\mu,
$$
we end up with
\begin{equation}\label{3.triv}
\frac d{dt} V_\eb(v(t))+\alpha(t)V_\eb(v(t))\le -\frac34\mu\|v(t)\|^2_H.
\end{equation}
Combining this inequality with the obvious estimate
\begin{equation}\label{3.vhm}
-\lambda_1^{-1}\|\xi\|^2_H\le V(\xi)\le \lambda_1^{-1}\|\xi\|^2_H
\end{equation}
and assuming that $\eb\le\frac12$, we prove the first formula of \eqref{3.difeq}.
\par
Let us prove the second inequality of \eqref{3.difeq}. Multiplying \eqref{3.gen} by $-4\eb A^{-1}v_+(t)$ and using that $(Av,v)\ge\lambda_N\|v\|^2$ if $v\in H_+$, we get
\begin{equation}
-2\eb\frac d{dt}\|v_+(t)\|^2\le 4\eb(\lambda_N+L)\|v\|^2_H.
\end{equation}
Multiplying now inequality \eqref{3.dcone} by $(1-\eb)>0$, taking a sum with the last inequality and fixing $\eb>0$ in such way that
$4\eb(\lambda_N+L)\le \frac14(1-\eb)\mu$, we end up with
\begin{equation}\label{3.triv1}
\frac d{dt} V_{-\eb}(v(t))+\alpha(t)V_{-\eb}(v(t))\le -\frac34(1-\eb)\mu\|v(t)\|^2_H.
\end{equation}
Using the inequality \eqref{3.vhm} again and assuming that $\eb<\frac14$, we end up with the desired second estimate of \eqref{3.difeq} and finish the proof of the lemma.
\end{proof}
The next corollary shows that the strong cone condition implies some kind of normal hyperbolicity in the sense that the trajectories outside of the cone squeeze stronger than the trajectories inside of the cone may expand.
\begin{corollary}\label{Cor3.hyp} Let the equation \eqref{3.gen} possesses the strong cone property, $T\in\R_+$ and $v(t)$, $t\in [-T,0]$ be a solution of problem \eqref{3.gen}. Then
\par
1) If $v(-T)\in K^+$, the whole trajectory $v(t)\in K^+$, $t\in[-T,0]$ and the following estimate holds:
\begin{equation}\label{11}
\|v(t)\|_{H^{-1}}^2\le Ce^{\bar\alpha(t)+\frac12\lambda_1 \mu t}\|v(0)\|^2_{H^{-1}},\ \ t\in[-T,0],
\end{equation}
where $\bar \alpha(t):=\int_{t}^0\alpha(s)\,ds$ and the constant $C$ is independent of $T$, $t$ and $v$.
\par
2) If $v(0)\notin K^+$, the whole trajectory $v(t)\notin K^+$, $t\in[-T,0]$ and the following estimate holds:
\begin{equation}
\label{12}
\|v(0)\|_{H^{-1}}^2\le Ce^{-\bar\alpha(t)+\frac12\lambda_1 \mu t}\|v(t)\|^2_{H^{-1}},\ \ t\in[-T,0],
\end{equation}
where the constant $C$ is independent of $T$, $t$ and $v$.
\end{corollary}
\begin{proof}Let $v(-T)\in K^+$. Then $V(v(-T))\le 0$ and from the cone condition \eqref{3.dcone} we conclude that $V(v(t))\le0$ for all $t\ge-T$ and $v(t)\in K^+$.
Integrating the second  inequality of \eqref{3.difeq}, we have
$$
V_{-\eb}(0)\le e^{-\bar \alpha (t)-\frac12\lambda_1\mu t}V_{-\eb}(t)
$$
and, therefore, since $V(v(t))\le0$, we obtain
\begin{multline}
\eb\|v(t)\|^2_{H^-1}\le \eb\|v(t)\|^2_{H^{-1}}-V(v(t))=-V_{-\eb}(v(t))\le\\\le e^{\bar\alpha(t)+\frac12\lambda_1 \mu t}(-V_{-\eb}(v(0))\le (1+\eb)e^{\bar\alpha(t)+\frac12\lambda_1 \mu t}\|v(0)\|^2_H
\end{multline}
and estimate \eqref{11} is proved.

Let now $v(0)\notin K^+$. Then, from the cone property \eqref{3.dcone} we conclude that $V(v(t))\ge0$ for all $t\in[-T,0]$. Integrating the first inequality of \eqref{3.difeq}, we get
$$
V_{\eb}(v(0))\le e^{-\bar\alpha(t)+\frac12\lambda_1\mu t}V_{\eb}(v(t))
$$
and, therefore, using that $V(v(t))\ge0$, we deduce
\begin{multline}
\eb\|v(0)\|^2_{H^{-1}}\le \eb\|v(0)\|^2_{H^{-1}}+V(v(0))=V_{\eb}(v(0))\le\\\le e^{-\bar\alpha(t)+\frac12\lambda_1 \mu t}V_{\eb}(v(t))\le (1+\eb)e^{-\bar\alpha(t)+\frac12\lambda_1 \mu t}\|v(t)\|_{H^{-1}}^2
\end{multline}
and the corollary is proved.
\end{proof}

The next corollary shows that the strong cone property in the differential form implies {\it both} cone and squeezing properties for the solutions of equation \eqref{3.gen}.

\begin{corollary} \label{cor_con_sq}
Let the equation \eqref{3.gen} possess the strong cone property \eqref{3.dcone} and let $v(t)$, $t\in[0,T]$ be a solution of \eqref{3.gen}.  Then the following properties are valid:
\par
1. Cone property (invariance of the cone $K^+$):
\begin{equation}\label{cone_con1}
v(0)\in K^+ \Rightarrow v(t) \in K^+, \text{ for all } t\ge 0.
\end{equation}
\par
2. Squeezing property: there exists positive $\gamma$ and $C$ such that
\begin{equation}\label{3.squeez_pr}
v(T) \not \in K^+ \Rightarrow \|v(t)\|_{H^{-1}} \le C e^{-\gamma t} \|v(0)\|_{H^{-1}}, t\in [0, T],
\end{equation}
where the constants $\gamma$ and $C$ are independent of $v$ and $T$.
\end{corollary}
\begin{proof} Indeed, the first assertion is an immediate corollary of inequality \eqref{3.dcone}. To verify the squeezing property, it is sufficient to use estimate \eqref{12} on the interval $[0,T]$ instead of $[-T,0]$. This together with inequality \eqref{3.squeez} give
$$
\|v(t)\|_{H^{-1}}^2\le Ce^{-\int_0^t\alpha(s)\,ds-\frac12\lambda_1\mu t}\|v(0)\|^2_{H^{-1}}\le Ce^{-(\alpha_-+\frac12\lambda_1\mu)t}\|v(0)\|_{H^{-1}}^2
$$
and the squeezing property is verified. Thus, the corollary is also proved.
\end{proof}
We are now ready to return to the non-linear equation \eqref{main_eq} and state for it the analogous strong cone condition in the differential form.

\begin{definition}\label{Def3.n-cone} Let the nonlinearity $F(u)$ satisfy the assumptions stated at the beginning of this section. We say that equation \eqref{main_eq} possesses a strong cone property in the differential form if there exist a positive constant $\mu$ and a bounded (Borel measurable) function $\alpha: H\to\R$ such that for every trajectory $u(t)\in H$, $t\in[0,T]$ of equation \eqref{main_eq} and every solution $v(t)$ of the equation of variations
\begin{equation}\label{3.eqvar1}
\Dt v+A(Av+F'(u(t))v)=0,
\end{equation}
the following analogue of \eqref{3.dcone} holds
\begin{equation}\label{3.dcone1}
\frac d{dt}V(v(t))+\alpha(u(t))V(v(t))\le-\mu\|v(t)\|^2_H
\end{equation}
and the function $\alpha$ satisfies
\begin{equation}
0<\alpha_-\le\alpha(u)\le\alpha_+<\infty.
\end{equation}
\end{definition}

The relation between the strong cone conditions for the non-linear equation \eqref{main_eq} and for the equation \eqref{3.dif} for the differences of its solutions is clarified in the following lemma.

\begin{lemma}\label{Lem3.con-con} Let equation \eqref{main_eq} possess the strong cone property in the sense of Definition \ref{Def3.n-cone}. Then, for every two solutions  $u_1(t)$ and $u_2(t)$, $t\in[T,S]$ of equation \eqref{main_eq},  the associated equation \eqref{3.dif} for the differences of solutions also possess the strong cone condition (in the sense of Definition \ref{Def3.diffcon}) with the same constant $\mu$ and with the constant $\alpha$ satisfying
\begin{equation}
\label{20}
\alpha(t)=\alpha_{u_1,u_2}(t):=\int_0^1\alpha(s u_1(t)+(1-s)u_2(t))\,ds.
\end{equation}
\end{lemma}
\begin{proof} Indeed, according to Lemma \ref{Lem3.point}, the strong cone condition for equation \eqref{main_eq} is equivalent to
\begin{multline}
-2(F'(u)w,w_--w_+)-2(Aw,w_--w_+)+\\+\alpha(u)(\|w_-\|^2_{H^{-1}}-\|w_+\|^2_{H^{-1}})\le-\mu\|w\|^2_H,\ \ \forall w\in H^1
\end{multline}
and every $u\in H$. Replacing $u=s u_1(t)+(1-s)u_2(t)$ in this inequality and integrating over $s\in[0,1]$, we end up with
\begin{multline}\label{23}
-2(l_{u_1,u_2}(t)w,w_--w_+)-2(Aw,w_--w_+)+\\+\alpha_{u_1,u_2}(t)(\|w_-\|^2_{H^{-1}}-\|w_+\|^2_{H^{-1}})\le-\mu\|w\|^2_H,\ \ \forall w\in H^1
\end{multline}
which according to Lemma \ref{Lem3.point} again is equivalent to the strong cone condition for equation \eqref{3.dif} and the lemma is proved.
\end{proof}
We summarize the obtained results in the following theorem which can be considered as the main result of the section.

\begin{theorem}\label{Th3.main} Let the nonlinearity $F(u)$ be globally bounded and satisfy \eqref{3.1} and \eqref{3.2} and let also the associated equation \eqref{main_eq} possess the strong cone condition in the sense of Definition \ref{Def3.n-cone}. Then, equation \eqref{main_eq} possess the $N$-dimensional Lipschitz continuous inertial manifold (see Definition \ref{Def1.IM}).
\end{theorem}
\begin{proof} Indeed, according to Lemma \ref{Lem3.con-con}, for any two solutions $u_1(t)$ and $u_2(t)$ of equation \eqref{main_eq}, the associated equation for differences of solutions \eqref{3.dif} possesses the strong cone condition. Then, due to Corollary \ref{cor_con_sq}, equation \eqref{3.dif} satisfies the cone and squeezing property. In particular, looking at the solution $v(t)=u_1(t)-u_2(t)$ of this equation, we see that equation \eqref{main_eq} possesses the cone and squeezing property in the sense of Definition \ref{Def1.con-squeez}. Thus, Theorem \ref{th_in_man} is applicable and gives the existence of the desired inertial manifold. This finishes the proof of the theorem.
\end{proof}
\begin{remark}\label{Rem3.smooth} Theorem \ref{Th3.main} simplifies the verification of the conditions for the IM existence. Indeed, according to this theorem, we only need to check {\it one} differential inequality \eqref{3.dcone1} for the equation of variations \eqref{3.eqvar1}. Note also that, analogously to Remark \ref{Rem2.smooth}, it is sufficient to verify all of the conditions for the trajectories $u$ satisfying \eqref{2.important} only.
\end{remark}
To conclude this section, we show that the classical spectral gap condition implies the strong cone condition and, therefore, guarantees the existence of the IM.
\begin{proposition}\label{Prop3.spectral}  Let the nonlinearity $F(u)$ be globally bounded and the following spectral gap condition be satisfied for some $N\in\mathbb N$:
\begin{equation}
\lambda_{N+1} - \lambda_N > L,
\end{equation}
where $L$ is a Lipschitz constant of the non-linearity $F$. Then for every two solutions $u_1(t)$ and
$u_2(t)$ of the equation \eqref{main_eq}, the associated
 equation \eqref{3.dif} for the differences possesses the strong cone condition with
\begin{equation}\label{3.am}
 \alpha(t)=2\alpha := 2\lambda_{N+1} \lambda_N\ \ \text{and}\ \ \mu := 2(\lambda_{N+1} - \lambda_N - L) > 0.
\end{equation}
Thus, equation \eqref{main_eq} possesses the Lipschitz IM.
\end{proposition}
\begin{proof}
Let $v(t)$ be a solution of \eqref{3.dif}. Then,
 multiplying equation for $v(t)$ first by $A^{-1}v_-(t)$, second by $-A^{-1} v_+(t)$ and taking sum of them we obtain:
\begin{multline}\label{3.cone-est}
\frac{1}{2} \frac{d}{dt}V(v(t)) + \alpha V(v(t)) = ((\alpha A^{-1} -A)v_-,v_-)+  \\ + ((A - \alpha A^{-1})v_+, v_+) - (l_{u_1(t),u_2(t)}v, v_- - v_+).
\end{multline}
By the definition of $\alpha$ and $\mu$, we have
\begin{multline}\label{3.one-est}
((A - \alpha A^{-1})v_+, v_+)=\sum_{n=1}^N(\lambda_n-\lambda_N\lambda_{N+1}\lambda_n^{-1})|v_n|^2\le\\\le \sum_{n=1}^N(\lambda_N-\lambda_{N+1})|v_n|^2=-(\lambda_{N+1}-\lambda_N)\|v_+\|^2_H.
\end{multline}
Analogously,
\begin{multline}\label{3.dva-est}
((\alpha A^{-1} -A)v_-,v_-)=\sum_{n=N+1}^\infty(\lambda_N\lambda_{N+1}\lambda_n^{-1}-\lambda_n)|v_n|^2\le\\\le \sum_{n=N+1}^\infty(\lambda_N-\lambda_{N+1})|v_n|^2=-(\lambda_{N+1}-\lambda_N)\|v_-\|^2_H.
\end{multline}
Inserting these estimates to \eqref{3.cone-est} and using that $\|l_{u_1(t),u_2(t)}\|_{\mathcal L(H,H)}\le L$, we have
\begin{equation}
\frac{1}{2} \frac{d}{dt}V(v(t)) + \alpha V(v(t)) \le -(\lambda_{N+1} -\lambda_N)\|v\|^2_H + L\|v\|^2_H = -\mu \|v\|^2_ H.
\end{equation}
Thus, the strong cone condition for \eqref{3.dif} is verified and the proposition is proved.
\end{proof}

\section{Smoothness of the inertial manifolds}\label{s3.5}

The aim of this section is to obtain the extra smoothness of the function $\Phi: H_+ \to H_-$ which determine the inertial manifold $\mathcal{M}$ under the assumption that non-linearity $F$ is $C^{1+\delta}(H,H)$ for some $\delta \in (0,1)$, i.e.,
\begin{equation}\label{3}
\|F(u_1)-F(u_2)-F'(u_1)(u_1-u_2)\|_H\le C\|u_1-u_2\|^{1+\delta}_H, \ \ u_1, u_2 \in H.
\end{equation}
To be more precise, the main result of this section is the following theorem.
\begin{theorem}\label{Th4.main} Let the assumptions of Theorem \ref{Th3.main} hold and let also the assumption \eqref{3} on $F$  be valid for some $\delta>0$. Then the map $\Phi$ is Frechet differentiable and $C^{1+ \varepsilon}$-smooth for some  $\varepsilon>0$, i.e,
\begin{equation}\label{28}
\|\Phi(u^1_+)-\Phi(u^2_+)-\Phi'(u^1_+)(u_+^1-u_+^2)\|_{H^{-1}} \le C\|u^1_+-u^2_+\|_H^{1+\eb}.
\end{equation}
\end{theorem}
\begin{proof} To verify \eqref{28}, we first need to study the Frechet derivative $\Phi'(u^1_+)$. Following the definition of $\Phi$, it is natural to expect that this derivative is defined as follows:
\begin{equation}
\label{29}
\Phi'(u^1_+)w_+:=\lim_{T\to\infty} Q_N w_T(0),
\end{equation}
where $u_+^i\in H_+$, $i=1,2$, and $w_T(t)$, $T>0$, solves
\begin{equation}
\label{30}
\pt w + A^2w + A F'(u^1(t))w = 0,\ \ w_+\big|_{t=0}=w_+,\ \ w_-\big|_{t=-T}=0.
\end{equation}
 Here and below $w_+ \in H_+$ and $u^i(t)$, $t\le0$, $i=1,2$, are the solutions of \eqref{main_eq} belonging to the inertial manifold $\mathcal M$ and satisfying $P_Nu^i(0)=u^i_+$.

{\it Step 1.} Well-posedness of $\Phi'(u^1(t))$. The existence of a solution for problem \eqref{30} can be verified exactly as in Theorem \ref{th_in_man} and to define the operator $\Phi'$, we only need to check the existence of the limit \eqref{29}. Indeed, since the trajectory  $w_T \in K^+$, according to Corollary \ref{Cor3.hyp}  and estimate \eqref{11}, we get
\begin{equation}\label{31}
\|w_T(t)\|_{H^{-1}}^2\le Ce^{\bar\alpha(t)+\frac12\lambda_1 \mu t}\|w_+\|_{H^{-1}}^2.
\end{equation}
Here $\bar\alpha(t)=\int_t^0\alpha(u^1(s))\,ds$ and $\mu$ is the same as in the strong cone inequality.
Consider now another approximation $w_{T_1}(t)$, $T_1\ge T$, and their difference $w_{T,T_1}(t):=w_{T_1}(t)-w_T(t)$. This trajectory does not belong to the cone $K^+$ at $t=0$ and, therefore, it is not in $K^+$ for all $t\in[-T,0]$. Using now \eqref{31} together with estimate \eqref{12} of Corollary \ref{Cor3.hyp}, we end up with
\begin{multline}\label{32}
\|w_{T,T_1}(0)\|^2_{H^{-1}}\le Ce^{-\bar\alpha(-T)-\frac12 \lambda_1 \mu T}\|w_{T,T_1}(-T)\|^2_{H^{-1}}\le\\\le C_1e^{-\bar\alpha(-T)-\frac12 \lambda_1 \mu T}(\|w_T(-T)\|^2_{H^{-1}}+\|w_{T_1}(T)\|^2_{H^{-1}})\le\\\le
C_2e^{- \lambda_1 \mu T}(\|w_T(0)\|^2_{H^{-1}}+\|w_{T_1}(0)\|^2_{H^{-1}})\le C_3e^{-\lambda_1\mu T}\|w_+\|^2_{H^{-1}}.
\end{multline}
Thus, $w_T(0)$ is a Cauchy sequence and the limit \eqref{29} exists and, therefore, the operator $\Phi'(u^1(t))$ is well-defined. Moreover, according to \eqref{31}, we have the following estimate for the limit function $w(t)$:
\begin{equation}\label{33}
\|w(t)\|_{H^{-1}}^2\le Ce^{\bar\alpha(t)+\frac12\lambda_1 \mu t}\|w_+\|^2_{H^{-1}},\ \ w_+\in H_+,\ \ t\le0.
\end{equation}

{\it Step 2.} Estimate for the difference $v(t):=u^1(t)-u^2(t)$.

Let $u^1(t)$ and $u^2(t)$ be two trajectories on the inertial manifold defined by the limit \eqref{lim} which correspond to the initial data $u^1_+\in H_+$ and $u_2^+\in H_+$ respectively and $w_+: = u_1^+ - u_2^+$. Since $v(t)\in K^+$ then, due to estimate \eqref{11} and the assumption that the exponent $\alpha(t)$ is globally bounded, we have
\begin{equation}\label{34}
\|v(t)\|_{H^{-1}}^2\le Ce^{-Kt}\|w_+\|^2_{H^{-1}},\ \ t\in \R_-.
\end{equation}
Our aim at this step is to improve \eqref{34} and to obtain the estimate which is analogous to \eqref{33}. To this end, we note that
the function $v$ solves the equation
\begin{equation}
\label{35}
\Dt v+A^2v + A F'(u^1(t))v + A [l_{u_1(t),u_2(t)}-F'(u_1(t))]v(t)=0,
\end{equation}
where $l_{u_1(t),u_2(t)}:=\int_0^1F'(u_1(t)+s v(t))\,ds$. Since $F$ satisfies \eqref{3} and $v(t)\in K^+$, we have
\begin{equation}\label{36}
\|l_{u_1(t),u_2(t)}-F'(u_1(t))\|_{\Cal L(H,H)}\le C\|v(t)\|_{H}^{\delta}\le Ce^{-K\delta t/2}\|w_+\|^\delta_H, \ \ t\in(-\infty,0].
\end{equation}
Thus, treating equation \eqref{35} as a non-homogeneous problem in the form of \eqref{9} with the right-hand side  $h(t):=[l_{u_1(t),u_2(t)}-F'(u_1(t))]v(t)$ and according to \eqref{10} and \eqref{3.difeq} (see also \eqref{3.triv1}), we get that, for a sufficiently small $\eb>0$, the following estimate holds:
\begin{multline}\label{37}
\frac d{dt}V_{-\eb}(v(t))+(\alpha(u^1(t))-\frac12\lambda_1\mu)V_{-\eb}(v(t))\le\\\le-\frac\mu4\|v(t)\|^2_H+C\|h(t)\|_H\|v(t)\|_H\le
(-\frac\mu4+Ce^{-K\delta t/2}\|w_+\|_H^\delta)\|v(t)\|_H^2\le 0
\end{multline}
if $-t\le\frac2{K\delta}\ln\frac {\mu}{4C \|w_+\|_H^{\delta}}$. Thus, since $v\in K^+$, analogously to \eqref{11}, we end up with the following estimate:
\begin{equation}
\label{38}
\|v(t)\|_{H^{-1}}^2\le Ce^{\bar\alpha(t)+\frac12\lambda_1 \mu t}\|w_+\|^2_{H^{-1}},\ \ w_+\in H_+,\ \ t\in[-T,0],\ \ T=\frac2{K\delta}\ln\frac {\mu}{4C\|w_+\|_H^{\delta}}
\end{equation}
which differs from \eqref{33} only by the presence of the lower bound for $t$.
\par
{\it Step 3.} Applying the parabolic smoothing property. Up to the moment, we have obtained estimates \eqref{33} and \eqref{38} for the $H^{-1}$ norms of the functions $v(t)$ and $w(t)$ only, but we need to control more regular norms of these functions in the sequel. To get this control, we remind that, analogously to Proposition \ref{Prop1.trivial}, we have the following parabolic smoothing property for the solutions of problem \eqref{30}:
\begin{equation}
\|w(t+1)\|_{H^{2-\kappa}}\le C_\kappa \|w(t)\|_{H^{-1}},
\end{equation}
where $\kappa>0$ is arbitrary and the constant $C_\kappa$ depends only on $\kappa$. Combining this estimate with \eqref{33}, we get
\begin{equation}\label{4.33}
\|w(t)\|_{H^{2-\kappa}}^2\le C_\kappa e^{\bar\alpha(t)+\frac12\lambda_1 \mu t}\|w_+\|^2_{H^{-1}},\ \ w_+\in H_+,\ \ t\le0.
\end{equation}
Analogously, applying the parabolic smoothing property to equation \eqref{3.dif} and using \eqref{38}, we get
\begin{equation}
\label{4.38}
\|v(t)\|_{H^{2-\kappa}}^2\le C_\kappa e^{\bar\alpha(t)+\frac12\lambda_1 \mu t}\|w_+\|^2_{H^{-1}},\ \ w_+\in H_+,\ \ t\in[-T,0]
\end{equation}
with $T=\max\{0,\frac2{K\delta}\ln\frac {\mu}{4C\|w_+\|_H^{\delta}}-1\}$.
\par
{\it Step 4.} Estimate for $\theta(t):= w(t)-v(t)$. This function solves the following equation:
\begin{equation}
\label{39}
\pt\theta+A^2 \theta + A (F'(u^1(t))\theta)+ A h(t)=0.
\end{equation}
Then, combining  \eqref{3.difeq} and \eqref{10}, for sufficiently small $\eb>0$, we have
\begin{equation}
\label{40}
\frac d{dt}V_\eb(\theta(t))+(\alpha(u^1(t))+\frac12\lambda_1\mu)V_\eb(\theta(t))\le C\|h(t)\|_H\|\theta(t)\|_H
\end{equation}
and using that $\theta(0)\not\in K^+$, analogously to \eqref{12}, we have
\begin{equation}\label{41}
\|\theta(0)\|^2_{H^{-1}}\le Ce^{-\bar\alpha(-T)+\frac12\lambda_1\mu T}\|\theta(-T)\|^2_{H^{-1}}+
C\int_{-T}^0e^{-\bar\alpha(t)+\frac12\lambda_1 \mu t}\|h(t)\|_H\|\theta(t)\|_H\,dt.
\end{equation}
Assume now that $T=\max\{0,\frac2{K\delta}\ln\frac {\mu}{4C\|w_+\|_H^{\delta}}-1\}$ and using estimates \eqref{4.33}, \eqref{34} and \eqref{4.38} (with $\kappa=2$) as well as \eqref{36}, we finally arrive at
\begin{multline}\label{42}
\|\theta(0)\|^2_{H^{-1}}\le Ce^{-\bar\alpha(-T)+\frac12 \lambda_1 \mu T}(\|v(-T)\|^2_{H^{-1}}+\|w(-T)\|^2_{H^{-1}})+\\+
C\int_{-T}^0e^{-\bar\alpha(t)+\frac12\lambda_1 \mu t}\|v(t)\|_H^\delta\|v(t)\|_H(\|w(t)\|_H+\|v(t)\|_H)\,dt\le
\\\le
Ce^{-\lambda_1\mu T}\|w_+\|^2_{H^{-1}}+C\|w_+\|_{H^{-1}}^{2+\delta/2}\int_{-T}^0e^{-(-\lambda_1\mu+K\delta/2) t}\,dt.
\end{multline}
Decreasing the exponent $\delta$ if necessary, we may assume that $-\lambda_1\mu+K\delta/2\le-\lambda_1\mu/2$ and therefore
\begin{equation}
\label{43}
\|\theta(0)\|^2_{H^{-1}}\le Ce^{-\lambda_1\mu T}\|w_+\|^2_{H^{-1}}+C\|w_+\|_{H^{-1}}^{2+\delta/2}\le C\|w_+\|^{2(1+\eb)}_{H^{-1}}
\end{equation}
for some $\eb=\eb(\delta,\mu)>0$. Thus, the desired estimate \eqref{28} is proved and the theorem is also proved.
\end{proof}

\begin{remark}\label{Rem4.good} Applying the parabolic smoothing property to the equation for $\theta(t)$, it is not difficult to verify the stronger version of estimate \eqref{28}, namely
\begin{equation}\label{4.28}
\|\Phi(u^1_+)-\Phi(u^2_+)-\Phi'(u^1_+)(u_+^1-u_+^2)\|_{H^{2-\kappa}} \le C_\kappa\|u^1_+-u^2_+\|_H^{1+\eb},
\end{equation}
where $\kappa>0$ is arbitrary and $C_\kappa$ depends only on $\kappa$. Moreover, as follows from the proof, estimate \eqref{3} is actually used for $u_1$ and $u_2$ satisfying \eqref{2.important} only and can be replaced by
\begin{equation}\label{4.3}
\|F(u_1)-F(u_2)-F'(u_1)(u_1-u_2)\|_H\le C\|u_1-u_2\|^{\delta}_{H^{2-\kappa}}\|u_1-u_2\|_{H}, \ \ u_1, u_2 \in H^{2-\kappa},
\end{equation}
for some $\kappa\in(0,2]$. As we will see below, assumption \eqref{4.3} is much easier to verify in applications than the initial assumption \eqref{3} which is more natural for the abstract theory.
\par
Note also that the result of Theorem \ref{Th4.main} is in a sense optimal since the typical regularity of the inertial manifolds is exactly $C^{1+\eb}$ for some small $\eb>0$. The further regularity ($C^2$ or more) requires essentially stronger spectral gap assumptions which are usually satisfied only in the case of small Lipschitz constant $L$, see \cite{kok,Zel} for more details.
\end{remark}

\section{Spatial averaging: an abstract scheme}\label{s4}

In this section, we adapt the method of spatial averaging developed in \cite{mal-par} to the class of abstract Cahn-Hilliard equations \eqref{main_eq}. To this end, we first need to introduce some projectors.
\par
Let $N \in \N$ and $k > 0$ be such that $\lambda_N > k$. Then,
\begin{equation}
P_{k,N}u:=\sum_{n:\, \lambda_n < \lambda_N -k}(u,e_n)e_n, \ \ Q_{k,N}u := \sum_{n:\, \lambda_n > \lambda_{N} +k}(u,e_n)e_n,
\end{equation}
and
\begin{equation*}
R_{k,N}u := \sum_{n:\, \lambda_N - k \le \lambda_n \le \lambda_{N} +k}(u,e_n)e_n.
\end{equation*}
As has been observed in \cite{mal-par} (at least on the level of reaction-diffusion equations, see also \cite{Zel}), the spectral gap condition is actually used only  for the control the norm of the "intermediate" part $R_{k,N}\circ F'(u)\circ R_{k,N}$ of the derivative $F'(u)$ where $k\sim L^2$.
Moreover, if this intermediate part is close to the scalar operator then the spectral gap condition may be relaxed.
The following theorem adapts this result to the case of the Cahn-Hilliard equations.

\begin{theorem}\label{th_manifold_ex}
Let the function $F$ be globally Lipschitz with the Lipschitz constant $L$, globally bounded and differentiable  and let the number $N$ be such that
\begin{equation}\label{est_middle_terms}
\|R_{k,N}\circ F'(u)\circ R_{k,N}v - a(u) R_{k,N}v\|_H \le \delta \|v\|_H, \ \ u, v \in H,
\end{equation}
where $a(u)\in \R$ is a scalar depending on $u$ and $\delta < L$. Assume also that
\begin{equation}\label{assump}
\frac{\theta}{2} >  \delta+\frac {2L k}{\lambda_N-k}+\frac{2L^2}{k-4L}+\frac{2 L^2\lambda_N }{(2 \lambda_N - k)k - 4L \lambda_{N}}, \ \  \lambda_N > 2L,
\end{equation}
as before $\theta = \lambda_{N+1} - \lambda_N$ and $k$ is chosen in such a way that
\begin{equation}
 (2\lambda_N - k)k> 4L \lambda_{N},\ \ k >4L.
\end{equation}
Then, equation \eqref{main_eq} possesses the strong cone property in the differential form and, consequently, there exists a Lipschitz N-dimensional inertial manifold for this equation.
\end{theorem}
\begin{proof}
Due to Theorem \ref{Th3.main}, we know that in order to prove the existence of inertial manifold  it is sufficient to check the validity of the strong cone inequality \eqref{3.dcone1}  for the equation of variations  \eqref{3.eqvar1} associated with the solution $u(t)$  of the Cahn-Hilliard equation \eqref{main_eq}.

To this end, we first  need the following estimate for the norm of $P_{k,N}w$ in $H^{-1}$ (here and below $\alpha=\lambda_N\lambda_{N+1}$):
\begin{multline} \label{P_N,k}
( (\alpha A^{-1} -A)w_+, w_+)=\sum_{n=1}^N(\lambda_N\lambda_{N+1}\lambda_n^{-1}-\lambda_n)|w_n|^2\ge\\\ge\sum_{n:\, \lambda_n<\lambda_N-k}(\lambda_N\lambda_{N+1}-\lambda_n^2)\lambda_n^{-1}|w_n|^2\ge\\\ge(\lambda_N^2-(\lambda_N-k)^2)\|P_{k, N} w\|^2_{H^{-1}}=(2 \lambda_N k - k^2)\|P_{k, N} w\|^2_{H^{-1}}
\end{multline}
and the similar estimate for the norm of $Q_{k,N}v$ in $H$
\begin{multline}\label{Q_N,k}
((A - \alpha A^{-1})w_-, w_-)=\sum_{n=N+1}^\infty(\lambda_n-\lambda_N\lambda_{N+1}\lambda_n^{-1})|w_n|^2\ge\\\ge\sum_{n:\,\lambda_n>\lambda_N+k}
(\lambda_n-\lambda_N\lambda_{N+1}\lambda_n^{-1})|w_n|^2\ge(\lambda_N+k-\lambda_N)\|Q_{k,N} w\|^2_H=k\|Q_{k,N} w\|^2_H.
\end{multline}
Let $w(t)$ be a solution of the equation of variations \eqref{3.eqvar1}. Then, arguing analogously to the derivation of
\eqref{3.cone-est} but using estimates \eqref{P_N,k} and \eqref{Q_N,k} together with \eqref{3.one-est} and \eqref{3.dva-est}, we get:
\begin{multline}\label{b_est_l(t)}
\frac{1}{2}\frac{d}{dt}V(w(t)) + \alpha V(w(t)) = ((\alpha A^{-1} -A)w_-,w_-) + ((A - \alpha A^{-1})w_+, w_+)- \\ -(F'(u)w, w_- - w_+) \le  -\frac{\theta}{2}\|v\|^2_H + \frac{1}{2}(((\alpha A^{-1} -A)w_-,w_-) + ((A - \alpha A^{-1})w_+, w_+)) - \\ - (F'(u)w, w_- - w_+) \le - \frac{\theta}{2}\|w\|^2_H - \frac{k}{2}\|Q_{k, N}w\|^2_H - \frac{(2 \lambda_N - k)k}{2} \|P_{k, N}w\|^2_{H^{-1}} - (F'(u)w, w_- - w_+).
\end{multline}
Estimating the last term, we have:
\begin{multline}\label{est_l(t)v,v--v+}
(F'(u)w, w_- - w_+) = (R_{k,N}\circ F'(u)w, w_- - w_+) +\\+((P_{k,N} + Q_{k,N})\circ F'(u)w, w_- -w_+)  = (R_{k,N}\circ F'(u) \circ R_{k,N}w, w_- - w_+) +\\+ (F'(u)w, Q_{k,N}w - P_{k,N}w) + (F'(u)\circ (Q_{k,N}w + P_{k,N}w), w_- - w_+) \ge \\ \ge (R_{k,N}\circ F'(u) \circ R_{k,N}w, w_- - w_+) - 2 L \|w\|_H (\|Q_{k,N}w\|_H +\|P_{k, N}w \|_{H})\ge\\\ge
(R_{k,N}\circ F'(u) \circ R_{k,N}w, w_- - w_+) - 2 L \|w\|_H (\|Q_{k,N}w\|_H + \lambda_N^{1/2}\|P_{k, N}w \|_{H^{-1}}).
\end{multline}
It would be convenient to define two more spectral projectors:
\begin{equation}
\tilde{P}_{k,N}w := \sum_{n:\, \lambda_N - k \le \lambda_n \le \lambda_N}(w,e_n)e_n \ \ \ \text{ and }\ \ \ \tilde{Q}_{k,N}w := \sum_{n:\, \lambda_{N+1} \le \lambda_n \le \lambda_{N} + k}(w,e_n)e_n.
\end{equation}
Then, obviously
\begin{multline}
\bigg|\lambda_N\|\tilde {P}_{k,N}v\|_{H^{-1}}^2-\|\tilde{P}_{k,N} v\|_H^2\bigg|\le\\\le
 \sum_{n:\, \lambda_N - k \le \lambda_n \le \lambda_N}|\lambda_N-\lambda_n|\lambda_n^{-1}|w_n|^2\le k\|\tilde {P}_{k,N}v\|_{H^{-1}}^2\le \frac{k}{\lambda_N-k}\|\tilde {P}_{k,N}v\|_{H}^2.
\end{multline}
and, analogously,
\begin{multline}
\bigg|\lambda_N\|\tilde {Q}_{k,N}v\|_{H^{-1}}^2-\|\tilde{Q}_{k,N} v\|_H^2\bigg|\le\\\le \sum_{n:\, \lambda_{N+1} \le \lambda_n \le \lambda_{N} + k}|\lambda_N\lambda_n^{-1}-1||w_n|^2\le \frac{k}{\lambda_N+k}\|\tilde{Q}_{k,N} v\|^2_H.
\end{multline}

Then, using  assumption \eqref{est_middle_terms}, we obtain:
\begin{multline}
(R_{k,N}\circ F'(u) \circ R_{k,N}w, w_- - w_+) \ge a(u)(\|\tilde{Q}_{k,N}w\|^2_H - \|\tilde{P}_{k,N}w \|^2_H) - \delta\|w\|^2_H \ge \\
\ge \lambda_N a(u)(\|\tilde{Q}_{k,N}w\|^2_{H^{-1}} - \|\tilde{P}_{k,N}w \|^2_{H^{-1}}) - \delta\|w\|^2_H -\\-|a(u)|\(\frac k{\lambda_N-k}\|\tilde {P}_{k,N}v\|_{H}^2+\frac{k}{\lambda_N+k}\|\tilde{Q}_{k,N} v\|^2_H\)
\ge \lambda_N a(u) V(w(t)) -\\ -  |a(u)|(\lambda_N\|P_{k,N}w\|^2_{H^{-1}} + \frac{\lambda_N}{\lambda_{N}+k}\|Q_{k,N}w\|^2_{H}) -
 (\delta+\frac k{\lambda_N-k}|a(u)|) \|w\|^2_H.
\end{multline}
Substituting this result into \eqref{est_l(t)v,v--v+} and using obvious fact that $|a(u)| \le L + \delta \le 2L$, we get:
\begin{multline} \label{5.est}
-(F'(u)w, w_- - w_+)\le -\lambda_N a(u) V(w) + 2 L( \lambda_{N}\|P_{k,N}w\|^2_{H^{-1}} + \|Q_{k,N}w\|^2_{H}) + \\
+2L\|w\|_H(\lambda_N^{1/2}\|P_{k,N}w\|_{H^{-1}} + \|Q_{k,N}w\|_{H})+(\delta+\frac {2L k}{\lambda_N-k}) \|w\|^2_H.
\end{multline}
Therefore, due to the Young inequality,
\begin{multline}\label{5.est1}
- \frac{\theta}{2}\|w\|^2_H - \frac{k}{2}\|Q_{k, N}w\|^2_H - \frac{(2 \lambda_N - k)k}{2} \|P_{k, N}w\|^2_{H^{-1}} - (F'(u)w, w_- -w_+)\le\\\le-\(\frac\theta2-\delta-\frac {2L k}{\lambda_N-k}\)\|w\|^2_H-\(\frac{(2 \lambda_N - k)k}{2}-2L\lambda_{N}\) \|P_{k,N}w\|^2_{H^{-1}}-\(\frac\kappa2-2L\)\|Q_{k,N}w\|^2_H+\\+2L\|w\|_H\(\lambda_N^{1/2}\|P_{k,N}w\|_{H^{-1}}+ \|Q_{k,N}w\|_{H}\)\le\\\le
-\(\frac\theta2-\delta-\frac {2L k}{\lambda_N-k}-\frac{2L^2}{\kappa-4L}-\frac{2 L^2\lambda_N }{(2 \lambda_N - k)k - 4L \lambda_{N}}\)\|w\|^2_H=-\frac\mu2\|w\|^2_H,
\end{multline}
where
$$
\frac\mu2:=\frac\theta2-\delta-\frac {2L k}{\lambda_N-k}-\frac{2L^2}{\kappa-4L}-\frac{2 L^2\lambda_N }{(2 \lambda_N - k)k - 4L \lambda_{N}}.
$$
Finally, inserting estimate \eqref{5.est1} into the right-hand side of \eqref{b_est_l(t)} and taking into account assumptions \eqref{assump} we see that the differential cone inequality \eqref{3.dcone1} is satisfied with the above $\mu$ and with
\begin{equation}
\alpha(u)= 2\alpha - 2\lambda_N a(u)> 2\lambda_{N}(\lambda_{N+1} - 2L)> 0.
\end{equation}
Thus, the desired strong cone condition  \eqref{3.dcone1} is proved and due to Theorem \eqref{Th3.main} equation \eqref{main_eq} possesses an $N-$dimensional inertial manifold.
\end{proof}

\begin{remark} The typical situation to apply the above proved theorem is when, for sufficiently small $\delta>0$ and any $k$ there exists an infinite sequence of $N\in\Bbb N$ such that
$$
\lambda_{N+1}-\lambda_N\ge\rho>0
$$
($\rho$ is independent of $N$ and $k$) such that the spatial averaging assumption \eqref{est_middle_terms} hold for every such $N$.  Then, for very large $N$, the main condition \eqref{assump} reads
\begin{equation}\label{5.spa-inf}
\frac\rho2>\delta+\frac{2L^2}{k-4L}+\frac{L^2}{k-2L}
\end{equation}
and we see that it is indeed satisfied if $\delta$ is small enough (say, $\delta<\frac\rho4$) and $k=k(\rho,L)$ is large enough (say,
$k=4L+\frac{12L^2}\rho$). This gives the existence of the desired inertial manifold for these large $N$s.
\end{remark}

As in the case of reaction-diffusion equations, see \cite{mal-par,Zel}, estimate  \eqref{est_middle_terms} is
too restrictive since the constant $\delta$ is uniform with respect to $u\in H$ and in applications it usually depends on the higher norms of $u$. Namely, similar to \cite{mal-par,Zel}, we give the following definition.

\begin{definition}\label{def_spat_av} We say that the non-linearity $F:H\to H$ satisfies the {\it spatial} averaging condition if it is globally bounded, Lipschitz continuous, differentiable in the sense that the mean value theorem \eqref{3.2} holds  and  there exist a positive exponent $\kappa$ and a positive constant $\rho$ such that for every $\delta > 0$, $R>0$ and $k >0$ there exists infinitely many values $N \in \N$ satisfying
\begin{equation}\label{5.agap}
\lambda_{N+1} - \lambda_N \ge \rho
\end{equation}
and
\begin{equation}\label{5.spa}
\sup_{\|u\|_{H^{2-\kappa}}\le R}\biggl\{ \|R_{k,N}\circ F'(u)\circ R_{k,N}v - a(u) R_{k,N}v\|_H \biggr\}\le \delta \|v\|_H,
\end{equation}
for some scalar multiplier $a(u) = a_{N,k,\delta}(u)\in \R$ which is assumed to be bounded  Borel measurable as a function from $H$ to $\R$.
\end{definition}
In this case, although we do not know how to construct the IM for the initial problem \eqref{main_eq}, it is possible to {\it modify} this equation outside of the absorbing ball in such a way that the new equation will possess the IM. Then, the obtained IM will still be invariant with respect to the solution semigroup $S(t)$ of the initial equation \eqref{main_eq} at least in the neighborhood of the global attractor $\Cal A$ and therefore will contain all of its non-trivial dynamics. By this reason, the IM for the modified equation is often referred as the IM for the initial problem \eqref{main_eq}, see e.g., \cite{FST,tem} for more details.
\par
To be more precise, according to the dissipative estimate \eqref{diss} with $s=2$ together with the smoothing property \eqref{smooth}, the set
\begin{equation}
\Cal B_2:=\big\{u\in H^2,\ \ \|u\|_{H^2}\le 2R_*\big\}
\end{equation}
is an absorbing ball for the solution semigroup $S(t)$ associated with the Cahn-Hilliard equation. 
 Let us introduce, following \cite{mal-par}, the cut-off function $\varphi(\eta) \in C^{\infty}(\R)$ such that:
 \begin{equation}
 \varphi(\eta) = 1 ,\ \ \eta\le(2 R_*)^2 \ \text{ and }\varphi(\eta)=\frac{1}{2}, \ \ \eta \ge R_1^2,
 \end{equation}
 where $R_1 > 2R_*$ and :
 \begin{equation}\label{phi_rest}
 \varphi'(\eta)\le 0 \ \ \text{ and }\ \ \frac{1}{2}\varphi(\eta) + \eta \varphi'(\eta)> 0 ,\  \eta \in \R.
 \end{equation}
 Thus, \eqref{phi_rest} gives us  the restriction
 \begin{equation}
\varphi(\eta)\ge \frac{2 R_*}{\sqrt{\eta}},\ \  \eta \ge 16R_*
 \end{equation}
 and, therefore, $R_1 \ge 4 R_*$.
\par
 Finally, for every $N \in \N$, we  introduce the following cut -off version of the problem \eqref{main_eq}:
 \begin{equation}\label{main_eq_cut}
 \pt u + A^2 u + A F(u) - A^2 P_N u + \varphi(\|A P_N u \|^2_H)A^2 P_N u = 0.
 \end{equation}
 Then, on the one hand, by the construction of the cut-off function $\varphi(\eta)$, we see that equation \eqref{main_eq_cut} coincides with \eqref{main_eq} inside the absorbing ball $\Cal B_2$ and, on the other hand, the following key result holds.

 \begin{theorem}\label{Th5.main}
 Let the non-linearity $F$ satisfy the spatial averaging assumption for some $\kappa\in(0,2)$. Then, there exist infinitely many $N$s such that the strong cone condition is satisfied  for the modified equation \eqref{aux_eq} and, thus, for every such $N$, it possesses an N-dimensional Lipschitz continuous IM.
 \end{theorem}
 \begin{proof} According to Theorem \ref{Th3.main}, we only need to verify the strong cone condition \eqref{3.dcone1} for the equation of variations associated with the modified equation \eqref{aux_eq}:
 \begin{equation}\label{5.var}
 \pt w + A^2 w + A (F'(u(t))w) = A^2 P_N w - T'(u(t))A w,
\end{equation}
 where $T(u):= \varphi (\|A P_N u\|^2_H)A P_N u$ and $u(t)$ is a solution of \eqref{aux_eq}.
 To this end, we need the following technical lemma originally proved in \cite{mal-par}.

 \begin{lemma}\label{lem_est_T(u)}
 Under the above assumptions the following estimate is valid:
 \begin{equation}\label{est_T(u)}
 ( T'(u)v,v) \le \frac{1}{2} \lambda_N \|v\|^2_H + \frac{1}{2}(A v, v),\ \  \forall v \in H_+,\  \ u\in H.
 \end{equation}
 \end{lemma}
 \begin{proof} For the sake of completeness we provide the simplified proof of the lemma following to \cite{Zel}.
 It is based on the following inequality:
 \begin{equation}\label{5.CS}
 2(v,y)(w,y) \ge \|y\|_H^2 ((v,w) - \|v\|_H \|w\|_H)
 \end{equation}
 for any 3 vectors $v,w,y\in H$. To verify it, we rewrite \eqref{5.CS} in the equivalent form
 $$
 \(v-2\frac{(v,y)}{\|y\|^2_H}y,w\)\le\|v\|_H\|w\|_H
 $$
 and this follows from the Cauchy-Schwartz inequality and the fact that the map $v\to v-2\frac{(v,y)}{\|y\|^2_H}y$ is a reflection with respect to the plane orthogonal to $y$ and, thus, is an isometry.
 \par
 Using now \eqref{5.CS} and inequalities \eqref{phi_rest}, we have that, for any $v \in H_+$,
 \begin{multline}
 (T'(u)v,v)= 2 \varphi'(\|Au_+\|^2_H)(A u_+, A v)(A u_+, v) + \varphi(\|A u_+\|^2_H)(A v,v)  \le \\ \le - \varphi'(\|Au_+\|^2_H) \|Au_+\|^2_H (\|Av\|_H \|v\|_H - (A v, v)) + \varphi(\|Au_+\|^2_H) (A v, v) \le \\ \le \frac{1}{2} \varphi(\|Au_+\|^2_H) \lambda_N \|v\|^2_H + \frac{1}{2} \varphi(\|Au_+\|^2_H)(A v, v)\le \frac{1}{2} (\lambda_N \|v\|^2_H + (Av,v))
 \end{multline}
 and the lemma is proved.
 \end{proof}

We are now ready to complete the proof of the theorem. Indeed, multiplying equation \eqref{5.var} by $w_--w_+$ acting as in the proof of Theorem \ref{th_manifold_ex} we come to the following equality:
 \begin{multline}\label{5.id}
 \frac{1}{2} \frac d{dt}V(w(t)) + \alpha V(w(t)) = ((\alpha A^{-1} - A)w_-, w_-) + ((A - \alpha A^{-1}) w_+, w_+) +\\+ (T'(u(t))w_+,w_+) - (A w_+,w_+) - (F'(u(t))w,w_--w_+),
 \end{multline}
where $\alpha=\lambda_N\lambda_{N+1}$.

  Then, using Lemma \eqref{lem_est_T(u)} and \eqref{3.dva-est}, we get:
\begin{multline}\label{same}
\frac{1}{2}\frac d{dt}V(w(t)) + \alpha V(w(t))\le -\frac{1}{2}(\lambda_{N+1} - \lambda_N )\|w_+\|^2_H +\frac{1}{2}((A - \alpha A^{-1})w_+,w_+)+\\ + ((\alpha A^{-1} - A)w_-, w_-) - (F'(u(t))w,w_--w_+)\le\\\le -\frac12(\lambda_{N+1}-\lambda_N)\|w\|^2_H+\frac{1}{2}((A - \alpha A^{-1})w_+,w_+)+\\ + \frac12((\alpha A^{-1} - A)w_-, w_-)- (F'(u(t))w,w_--w_+).
\end{multline}
 Fix an arbitrary point $t \ge 0$ and assume first that
 \begin{equation}\label{f_as}
  \|A P_N u(t)\|_H \le R_1,
 \end{equation}
where $R_1$ is the same as in the definition of the cut-off function $\varphi$.
Thus, we see that the structure of \eqref{same} is exactly the same as the structure of \eqref{b_est_l(t)}. In addition, due to \eqref{diss_Q_N}, we may assume without loss of generality that
\begin{equation}\label{5.kappa}
\|Q_N u(t)\|_{H^{2-\kappa}}\le 2R_\kappa,
\end{equation}
see Remarks \ref{Rem2.smooth} and \ref{Rem3.smooth}. Therefore,
 assumption \eqref{f_as} implies that
\begin{equation}
\|u(t)\|_{H^{2-\kappa}}\le \lambda^{-\frac{\kappa}{2}}_1 \|P_N u(t)\|_{H^2} + \|Q_N u(t)\|_{H^{2 - \kappa}} \le \lambda^{-\frac{\kappa}{2}}_1 R_1 + 2 R_\kappa \le R,
\end{equation}
 where $R$ is independent of the choice of $N$. Hence, using spatial averaging assumption \eqref{5.spa} (with this value of the parameter $R$, sufficiently small $\delta$ and sufficiently large $k$ in order to satisfy \eqref{5.spa-inf}) and repeating word by word the proof of Theorem \eqref{th_manifold_ex}, we conclude that there exist a sequence of $N$s such that
 \begin{equation}\label{5.good}
 \frac{1}{2} \frac{d}{dt}V(w(t)) + \alpha(u(t)) V(w(t)) \le -\mu \|w(t)\|^2_H,
 \end{equation}
 where $\alpha(u):=\lambda_N\lambda_{N+1}-\lambda_N a(u)$ and $\mu>0$ is independent of $u$, $N$ and $w$. Thus, we have verified the strong cone condition \eqref{3.dcone1} in the case when \eqref{f_as} is satisfied.
 \par
 Let us now consider the opposite case
 \begin{equation}\label{5.big}
  \|A P_N u(t)\|_H \ge R_1.
 \end{equation}
 The situation here is  much simpler. Indeed, instead of estimate \eqref{est_T(u)}, we may use better identity $(T'(u(t))w,w) = \frac{1}{2}(Aw_+,w_+)$.  Then, using the Lipschitz continuity of $F$ and the fact that both $(\alpha A^{-1} - A)Q_N$ and $(A - \alpha A^{-1})P_N$ are negatively definite (see  estimates \eqref{3.one-est} and \eqref{3.dva-est}), we transform identity \eqref{5.id} as follows:
 \begin{multline}
 \frac{1}{2} \frac{d}{dt} V(w(t)) + \alpha V(w(t)) \le\\\le ((\alpha A^{-1} - A)w_-, w_-) + ((A - \alpha A^{-1}) w_+, w_+)+L\|w\|^2_H-\frac12(Aw_+,w_+)\le\\\le
 L\|w\|^2_H-\frac12\alpha\|w_+\|^2_{H^{-1}}+((\alpha A^{-1} - A)w_-, w_-)\le\\\le
 L\|w\|^2_H+\frac14\alpha\(\|w_-\|^2_{H^{-1}}-\|w_+\|^2_{H^{-1}}\)-\frac14\lambda_{N+1}\|w_+\|^2_H-\frac14\|w_-\|_{H^1}^2\le\\\le
  \frac{1}{4} \alpha V(w(t)) + \(L-\frac{1}{4} \lambda_{N+1}\)\|w\|^2_H.
\end{multline}
Thus, if $ L < \frac{1}{4} \lambda_{N+1}-\mu$ ($\mu$ is the same as in \eqref{5.good}), we end up with estimate \eqref{5.good} with
 $\alpha(u)=\frac34\lambda_N\lambda_{N+1}$. Therefore
 the desired strong cone estimate \eqref{3.dcone1} is verified for the case when \eqref{5.big} is satisfied as well and the theorem is proved.
 \end{proof}

\begin{corollary}\label{Cor5.smooth} Let the assumptions of Theorem \ref{Th5.main} hold and let, in addition, the nonlinearity $F$ be smooth in the sense that assumption \eqref{4.3} hold for some $\kappa$ and $\delta>0$. Then, the inertial manifolds $\Cal M=\Cal M_N$ of the modified equations \eqref{main_eq_cut} are $C^{1+\eb}$-smooth for some $\eb=\eb_N>0$.
\end{corollary}
Indeed, the statement of the corollary follows immediately from the proved theorem and Theorem \ref{Th4.main}, see also Remark \ref{Rem4.good}.

\section{Inertial manifolds for the classical Cahn-Hilliard equation}\label{s5}
In this concluding section, we apply the developed abstract theory to the classical Cahn-Hilliard equation \eqref{0.CH} in the 3D case endowed by periodic boundary conditions. First, we need to embed this equation into the functional model \eqref{main_eq}. To this end, keeping in mind the conservation law \eqref{0.int} and our agreement that $\<u(t)\>=0$, we introduce the space
\begin{equation}
H:=L^2(\Bbb T^3)\cap\{\<u\>=0\}
\end{equation}
and the operator $A=-\Delta_x$ with the domain $D(A)=H^2(\Bbb T^3)\cap\{\<u\>=0\}$. Then, $H^s=H^s(\Bbb T^3)\cap\{\<u\>=0\}$, $s\in\R$, where $H^s(\Bbb T^3)$ is a Soblolev space of $2\pi$ periodic functions from $\R^3$ to $\R$. Moreover, any function $u\in H$ can be splitted into the Fourier series
\begin{equation}\label{6.f}
u(x)=\sum_{l\in\Bbb Z^3,\, l\ne0} u_l e^{i l.x},\ \ u_l=\frac1{(2\pi)^3}\int_{\Bbb T^3} u(x)e^{i l.x}.
\end{equation}
Here and below $l.x:=\sum_{i=1}^3l_ix_i$ is a usual inner product in $R^3$ and $|l|^2:=l.l$. The eigenvalues of the operator $A$ are naturally parametrised by the points of the lattice $l\in\Bbb Z^3\backslash\{0\}$, i.e.,
\begin{equation}\label{6.e}
A e_l=\lambda_l e_l,\ \ e_l=e^{i l.x},\ \ \lambda_l=|l|^2,\ \ l\in\Bbb Z^3\backslash\{0\}.
\end{equation}
Thus, due to the Parseval equality, the norm in the space $H^s$, $s\in\R$ is given by
\begin{equation}\label{6.n}
\|u\|_{H^s}^2=\sum_{l\in\Bbb Z^3,\, l\ne0}|l|^{2s} |u_l|^2,\ \ u(x)=\sum_{l\in\Bbb Z^3,\, l\ne0} u_l e^{i l.x}.
\end{equation}
We now return to equation \eqref{0.CH}. We assume that the non-linearity $f\in C^3(\R,\R)$ and is globally bounded together with its first derivative:
\begin{equation}\label{6.b}
1. \ |f(u)|+|f''(u)|\le K,\ \ 2.\ \ |f'(u)|\le L
\end{equation}
for all $u\in\R$. As we have already mentioned, \eqref{6.b} is not a big restriction since in the general case of  dissipative non-linearities $f$ (e.g., $f(u)=u-u^3$), we usually have an absorbing ball in $H^2\subset C(\Bbb T^3)$, making the proper cut-off of the non-linearity outside of the absorbing ball if necessary, we may assume without loss of generality that \eqref{6.b} is satisfied.
Finally, we introduce the non-linearity $F: H\to H$ via
\begin{equation}\label{6.non}
F(u)(x):=f(u(x))-\<f(u)\>,\ \ u\in H,
\end{equation}
where the last term is introduced in order to guarantee that $\<F(u)\>=0$ and, therefore, $F(u)\in H$.
\par
Then, it is immediate to see that equation \eqref{0.CH} is equivalent to equation \eqref{main_eq}, so the desired functional model is constructed. At the next step, we check that the assumptions stated above for the abstract equation \eqref{main_eq} are satisfied in this concrete case. First, from \eqref{6.non}, it is immediate to check that the non-linearity $F(u)$ is globally bounded and globally Lipschitz continuous in $H$. The next key proposition checks that the spatial averaging assumption is also satisfied.

\begin{proposition}\label{Prop6.av} Let the non-linearity $f\in C^3(\R,\R)$. Then, the spatial averaging assumption is satisfied for the operator $F$ for all $\kappa\in(0,\frac12)$ and
\begin{equation}
a(u):=\<f'(u)\>.
\end{equation}
\end{proposition}
\begin{proof} Note that $F'(u)v=f'(u)v-\<f'(u)v\>$, $u,v\in H$ and
\begin{equation}
R_{k,N}\circ F'(u)\circ R_{k,N}v=R_{k,N}\(f'(u)R_{k,N}v\),
\end{equation}
so the spatial averaging assumption for our case {\it coincides} with the analogous assumption for reaction-diffusion equations. By this reason, the proof of the proposition follows word by word to the one given in \cite{mal-par,Zel} for the case of reaction-diffusion equations. For the convenience of the reader, we sketch this proof below. It is based on the following non-trivial result from the number theory.

\begin{lemma}\label{Lem6.sp} Let
\begin{equation}
\Cal C^k_N:=\{l\in\Bbb Z^3\,:\, N-k\le|l|^2\le N+k\},\ \ \Cal B_r:=\{l\in\Bbb Z^3\,:\, |l|\le r\}.
\end{equation}
Then, for every $k>0$ and $r>0$, there exist infinitely many $N\in\Bbb N$ such that
\begin{equation}\label{6.spav}
\(\Cal C^k_{N}-\Cal C^k_N\)\cap \Cal B_r=\{0\}.
\end{equation}
\end{lemma}
The proof of this lemma is given in \cite{mal-par}.
\par
We are now ready to verify the spatial averaging principle for the non-linearity $f$. To this end, we first note that according to the Weyl asymptotics
$$
\lambda_N\sim CN^{2/3},
$$
so without loss of generality, we may replace the projector $R_{k,N}$ by the projector to the Fourier modes belonging to $\Cal C^k_N$ which (in slight abuse of notations), we also denote it by $R_{k,N}$, so we use below the definition
$$
(R_{k,N}v)(x):=\sum_{l\in\Cal C^k_N} v_l e^{il.x}.
$$
Denote $w(x):=f'(u(x))$. Then, the multiplication $w(x)v(x)$ is a convolution in Fourier modes
\begin{equation}
[wv]_m=\sum_{l\in\Bbb Z^3} w_{m-l}v_l
\end{equation}
and, due to condition \eqref{6.spav},
\begin{equation}
R_{k,N}\((w-\<w\>) R_{k,N}v\)=R_{k,N}\(w_{>r}R_{k,N}v\),
\end{equation}
where $w_{>r}(x):=\sum_{|l|>r}w_le^{il.x}$. Therefore,
\begin{equation}
\|R_{k,N}\((w-\<w\>) R_{k,N}v\)\|_H\le\|\(w_{>r}R_{k,N}v\)\|_H\le \|w_{>r}\|_{L^\infty}\|v\|_H.
\end{equation}
Furthermore, due to the interpolation, for $\kappa<1/2$,
$$
\|w_{>r}\|_{L^\infty}\le C\|w_{>r}\|_H^{1-\theta}\|w_{>r}\|_{H^{2-\kappa}}^\theta\le C r^{-(1-\theta)(2-\kappa)}\|w\|_{H^{2-\kappa}},
$$
where $\theta=\frac3{2(2-\kappa)}$. Finally, using that $H^{2-\kappa}$ is an algebra for $\kappa<\frac12$ and that $f'\in C^2$, we have 
\begin{equation}\label{6.est}
\|R_{k,N}\((w-\<w\>) R_{k,N}v\)\|_H\le C r^{-(1-\theta)(2-\kappa)}Q(\|u\|_{H^{2-\kappa}})\|v\|_H
\end{equation}
for some monotone increasing function $Q$. Thus, the right-hand side of \eqref{6.est} can be made arbitrarily small by increasing $r$, so estimate \eqref{5.spa} indeed holds with
$$
a(u)=\<w\>=\frac1{(2\pi)^3}\int_{\Bbb T^3}f'(u(x))\,dx.
$$
Since the eigenvalues of $A$ are integers, assumption \eqref{5.agap} also holds with $\rho=1$ and the nonlinearity $F(u)$ satisfies indeed the spatial averaging assumption. Thus, the proposition is proved.
\end{proof}
The next simple proposition shows that the map $F$ is smooth if $f$ is smooth.

\begin{proposition}\label{Prop6.sm} Let the nonlinear function $f\in C^2$ and satisfy \eqref{6.b}. Then, the non-linear operator $F(u)$ defined by \eqref{6.non} satisfies estimate \eqref{4.3} for any $\delta\in[0,1]$ and any $\kappa\in(0,\frac12)$.
\end{proposition}
\begin{proof} We first note that it is sufficient to verify \eqref{4.3} for $\delta=0$ and $\delta=1$ only. For simplicity, we verify estimate \eqref{4.3} for the first term $f(u(x))$ in the definition \eqref{6.non} of the nonlinearity $F(u)$ only. The estimate for the remaining term $\<f(u)\>$ can be obtained analogously. Let first $\delta=0$ and $u_1,u_2\in H$. Then, since $f$ is globally Lipschitz continuous,
$$
\|f(u_1)-f(u_2)-f'(u_1)(u_1-u_2)\|_H\le\|f(u_1)-f(u_2)\|_H+\|f'(u_1)(u_1-u_2)\|_H\le 2L\|u_1-u_2\|_H
$$
and estimate \eqref{4.3} is verified for that case. Let now $\delta=1$ and $u_1,u_2\in H^{2-\kappa}$. Then, according to the integral mean value theorem
\begin{multline}
f(u_1)-f(u_2)-f'(u_1)(u_1-u_2)=\int_0^1[f'(u_1+(1-s)(u_1-u_2))-f'(u_1)]\,ds(u_1-u_2)=\\=
\int_0^1\int_0^1f''(s_1u_1+s_1(1-s)(u_1-u_2))\,ds\,ds_1(u_1-u_2)^2
\end{multline}
and, due to assumption \eqref{6.b} and the embedding $H^{2-\kappa}\subset C$,
\begin{multline}
\|f(u_1)-f(u_2)-f'(u_1)(u_1-u_2)\|_H\le K\|(u_1-u_2)(u_1-u_2)\|_H\le\\\le K\|u_1-u_2\|_{L^\infty}\|u_1-u_2\|_H\le CK\|u_1-u_2\|_{H^{2-\kappa}}\|u_1-u_2\|_H.
\end{multline}
Therefore, estimate \eqref{4.3} holds for this case as well and the proposition is proved.
\end{proof}
Thus, all abstract assumptions from the previous sections are verified and we have proved the following theorem which is the main result of the paper.
\begin{theorem} Let the non-linear function $f\in C^3(\R,\R)$ and satisfy assumptions \eqref{6.b}. Then, there exists an infinite sequence of $Ns$ such that the classical Cahn-Hilliard problem \eqref{0.CH} on a 3D torus $\Bbb T^3=[-\pi,\pi]^3$ possesses an $N$ dimensional inertial manifold containing the global attractor. Moreover, these inertial manifolds are $C^{1+\eb}$-smooth for some $\eb=\eb_N>0$.
\end{theorem}
Indeed, the existence of IMs follows from Theorem \ref{Th5.main} and Proposition \ref{Prop6.av} and its smoothness follows from Theorem \ref{Cor5.smooth} and Proposition \ref{Prop6.sm}.



\end{document}